\documentclass[12pt,reqno]{amsart}

\usepackage{amsmath, amsthm, amssymb, stmaryrd}

\usepackage{enumerate}
\usepackage{url}
\usepackage{xcolor}  

\usepackage{parskip}

\usepackage{graphicx}

\makeatletter
\@namedef{subjclassname@2020}{%
  \textup{2020} Mathematics Subject Classification}
\makeatother

\let\OLDthebibliography\thebibliography
\renewcommand\thebibliography[1]{
  \OLDthebibliography{#1}
  \setlength{\itemsep}{0pt plus 0.3ex}
}

\usepackage{mathtools}

\usepackage{multirow, array}
\usepackage{placeins}
\usepackage{mathrsfs}

\usepackage{lipsum}
\usepackage{setspace}

\usepackage{hyperref}

\advance \topmargin by -\headheight
\advance \topmargin by -\headsep

\evensidemargin \oddsidemargin
\newtheorem{thm}{Theorem}[section]
\newtheorem{lemma}[thm]{Lemma}

\newtheorem{conj}[thm]{Conjecture}

\theoremstyle{definition}

\theoremstyle{remark}
\newtheorem{remark}[thm]{Remark}

\numberwithin{equation}{section}

\newcommand*\wrapletters[1]{\wr@pletters#1\@nil}
\def\wr@pletters#1#2\@nil{#1\allowbreak\if&#2&\else\wr@pletters#2\@nil\fi}

\def\alp{{\alpha}} 
\def\bet{{\beta}}  
\def\gam{{\gamma}} 
\def\del{{\delta}}

\def\kap{{\kappa}}
\def\lam{{\lambda}}

\def\ome{{\omega}}  
\def\eps{\varepsilon}

\def\le{\leqslant} \def\ge{\geqslant}  
\def \leq {\leqslant} \def \geq {\geqslant}

\def\d{{\,{\rm d}}}

\def \bN {\mathbb N}

\def \bR {\mathbb R}
\def \bZ {\mathbb Z}

\def \bx {\mathbf x}

\def \bzero {\mathbf 0}

\def \balp {{\boldsymbol{\alp}}}

\def \bgam {{\boldsymbol{\gam}}}
\def \bdel {{\boldsymbol{\del}}}

\def \cB {\mathcal B}
\def \cC {\mathcal C}

\begin{document}

\title[Counting multiplicative approximations]
{Counting multiplicative approximations}

\author[Sam Chow]{Sam Chow}
\address{Mathematics Institute, Zeeman Building, University of Warwick, Coventry CV4 7AL, United Kingdom}
\email{sam.chow@warwick.ac.uk}

\author[Niclas Technau]{Niclas Technau}
\address{Department of Mathematics, California Institute of Technology, 1200 E California Blvd., Pasadena, CA 91125,
USA}
\email{ntechnau@caltech.edu}

\thanks{}
\date{}

\maketitle

\begin{abstract} 
A famous conjecture of Littlewood (c. 1930) concerns approximating two real numbers by rationals of the same denominator, multiplying the errors. In a lesser-known paper, Wang and Yu (1981) established an asymptotic formula for the number of such approximations, valid almost always. Using the quantitative Koukoulopoulos--Maynard theorem of Aistleitner--Borda--Hauke, together with bounds arising from the theory of Bohr sets, we deduce lower bounds of the expected order of magnitude for inhomogeneous and fibre refinements of the problem.
\end{abstract}

\section{Introduction}

Khintchine's theorem \cite{Khi1924} is the foundational result of metric diophantine approximation. For $d \in \bN$, we denote by $\mu_d$ the $d$-dimensional Lebesgue measure. For $x \in \bR$, we write $\| x \| = \inf_{m \in \bZ} |x-m|$. The abbreviation i.o. stands for `infinitely often'. Throughout, let $k \ge 2$ be an integer, and let $\psi: \bN \to [0,1/2]$.

\begin{thm} [Variant of Khintchine, 1924]\label{thm: Khintchine} If $\psi$ is non-increasing then 
\[
\mu_1(\{ \alp \in [0,1]:  \| n \alp \| < \psi(n) \quad \mathrm{i.o.} \} )
=
\begin{cases}
1, &\text{if } \sum_{n=1}^\infty \psi(n) = \infty \\
0, &\text{if } \sum_{n=1}^\infty \psi(n) < \infty.
\end{cases}
\]
\end{thm}

Gallagher's theorem \cite{Gal1962} is one of the standard generalisations of Khintchine's theorem, and is related to a famous conjecture of Littlewood. For $d \in \bN$ and $\alp_1, \ldots, \alp_d \in \bR$, write $\balp = (\alp_1,\ldots,\alp_d)$. 

\begin{thm}[Gallagher, 1962]\label{thm: Gallagher}
If $\psi$ is non-increasing then
\begin{align*}
&\mu_k (\{ \balp \in [0,1]^k:  \| n \alp_1 \| \cdots \| n \alp_k \| < \psi(n) \quad \mathrm{i.o.}\}) \\
&= \begin{cases} 1, &\text{if } \sum_{n=1}^\infty \psi(n) (\log n)^{k-1} = \infty \\
0, &\text{if } \sum_{n=1}^\infty \psi(n) (\log n)^{k-1} < \infty.
\end{cases}
\end{align*}
\end{thm}

\begin{conj}[Littlewood, c. 1930]
If $\alp, \bet \in \bR$ then
\[
\liminf_{n \to \infty} n \| n \alp \| \cdot \| n \bet \| = 0.
\]
\end{conj}

As well as having infinitely many good rational approximations, one might be interested in the number of such approximations up to a given height. Schmidt~\cite{Sch1960} demonstrated such a result, see \cite[Theorem 4.6]{Har1998}. 

\begin{thm} [Variant of Schmidt, 1960] 
\label{thm: Schmidt}
For $N \in \bN$ and $\balp \in \bR^k$, 
denote by $S(\balp,N)$ the number of $n \in \bN$ such that
\[
n \le N, \qquad \| n \alp_i \| < \psi(n) 
\quad (1 \le i \le k).
\]
Assume that $\psi$ is non-increasing, assume that
\[
\Psi_k(N) := \sum_{n \le N} (2 \psi(n))^k 
\]
is unbounded, and let $\eps > 0$. Then, for almost all $\balp \in \bR^k$, we have
\[
S(\balp,N) = \Psi_k(N) + O_{k,\eps}(\sqrt{\Psi_k(N)} (\log \Psi_k(N))^{2+\eps})
\qquad (N \to \infty).
\]
\end{thm}

Wang and Yu \cite{WY1981} established a counting version of Gallagher's theorem. We state a variant of this below, deducing it from \cite[Theorem 4.6]{Har1998} in the appendix. For $N \in \bN$ and $\balp, \bgam \in \bR^k$, 
denote by $S^\times_\bgam(\balp,N,\psi)$ the number of $n \in \bN$ 
satisfying
\begin{equation}\label{eq: Littlewood product small}
n \le N, \qquad \| n \alp_1 - \gam_1 \| \cdots \| n \alp_k - \gam_k \| < \psi(n).
\end{equation}
For $N \in \bN$, define
\[
\Psi_k^\times(N) = \frac{1}{(k-1)!} \sum_{n \le N} \psi(n) (- \log(2^k \psi(n)))^{k-1}
\]
and
\[
\tilde \Psi_k^\times(N) = \sum_{n \le N} \psi(n) (\log n)^{k-1}.
\]
In our definition of $\Psi_k^\times(N)$, we adopt the convention that
\[
x (- \log x)^d \mid_{x = 0} \: = 0 \qquad (d \in \bR).
\]

\begin{remark}
In standard settings, we have
\[
-\log \psi(n) \asymp \log n
\qquad (n \ge 2),
\]
and correspondingly 
\[
\Psi_k^\times(N) \asymp \tilde \Psi_k^\times(N).
\]
\end{remark}

\begin{thm} [Variant of Wang--Yu, 1981] \label{WangYu} 
Assume that $\psi$ is non-increasing, that $\psi(n) \to 0$ as $n \to \infty$, and that $\Psi_k^\times(N)$ is unbounded. Then, uniformly for almost all $\balp \in \bR^k$, we have
\[
S^\times_\bzero(\balp,N,\psi) \sim \Psi_k^\times (N) 
\qquad (N \to \infty).
\]
\end{thm}

\begin{remark} Without the assumption that $\psi(n) \to 0$, there is a less explicit asymptotic main term, namely $T_k(N)$ as defined in the appendix. As can be seen from the proof therein, the assumption that $\psi(n) \to 0$ is necessary for Theorem \ref{WangYu} as stated.
\end{remark}

In this note, we address natural inhomogeneous and fibre refinements of this problem, as popularised by Beresnevich--Haynes--Velani \cite{BHV2020}. Our findings are enumerative versions of some of our previous results \cite{Cho2018, CT2019, CT2021}. 

\begin{thm} \label{MainThm} Let $\bgam \in \bR^k$ with $\gam_k = 0$, and let $\kap > 0$. Assume that $\psi$ is non-increasing, that $\psi(n) < n^{-\kap}$ for all $n$, and that $\tilde \Psi_k^\times(N)$ is unbounded. Then, for almost all $\balp \in \bR^k$, we have
\[
S^\times_\bgam(\balp,N,\psi) \gg \tilde \Psi_k^\times (N) \qquad (N \to \infty).
\]
The implied constant only depends on $k$.
\end{thm}

The \emph{multiplicative exponent} of $\balp \in \bR^d$ is
\[
\ome^\times(\balp) =
\sup \{ w: \| n \alp_1 \| \cdots \| n \alp_d \| < n^{-w} \quad \mathrm{i.o.} \}.
\]
Specialising $k = d$ and $\psi(n)=
(n (\log n)^{d+1})^{-1}$ 
in Gallagher's Theorem \ref{thm: Gallagher}, we see that $\ome^\times(\balp)=1$
for almost all $\balp\in \bR^d$. Thus, Theorem \ref{MainThm} is implied by the following fibre statement.

\begin{thm} \label{FibreThm} Let $\kap > 0$. Assume that $\psi$ is non-increasing, that $\psi(n) < n^{-\kap}$ for all $n$, and that $\tilde \Psi_k^\times(N)$ is unbounded. Let $\gam_1,\ldots,\gam_{k-1} \in \bR$, and suppose $(\alp_1,\ldots,\alp_{k-1})$ has multiplicative exponent $w <  \frac{k-1}{k-2}$, where $\frac{k-1}{k-2} \big |_{k=2} = \infty$. Then, for almost all $\alp_k$, we have
\[
S^\times_{(\gam_1,\ldots,\gam_{k-1},0)}((\alp_1,\ldots,\alp_k),N,\psi) \gg \tilde \Psi_k^\times (N) \qquad (N \to \infty).
\]
The implied constant only depends on $k, w,$ and $\kap$.
\end{thm}

A natural strategy to prove
Theorem \ref{FibreThm} is 
to isolate the metric parameter $\alp_k$ 
to one side of the inequality 
\eqref{eq: Littlewood product small}. Indeed, defining
\begin{equation} \label{PhiDef}
\Phi(n)= \frac{\psi(n)}{\| n \alp_1 - \gam_1 \| \cdots 
\| n \alp_{k-1} - \gam_{k-1} \|},
\end{equation}
the quantity
\[
S^\times_{(\gam_1,\ldots,\gam_{k-1},0)}((\alp_1,\ldots,\alp_k),N,\psi)
\]
counts positive integers $n\leq N$ satisfying
$\| n \alp_{k} \| < \Phi(n)$. If $\Phi$ were monotonic, then one could try to apply
Theorem \ref{thm: Schmidt}. The basic problem with this approach is that $\Phi$
is far from being monotonic.
Khintchine's theorem is false without the monotonicity assumption, as was shown by Duffin and Schaeffer~\cite{DS1941}. They proposed a modification of it, not requiring monotonicity of the approximating function, that was open for almost 80 years and only recently settled
by Koukoulopoulos and Maynard \cite{KM2020}.
We rely heavily on a very recent quantification 
by Aistleitner, Borda, and Hauke. Recall that $\psi: \bN \to [0,1/2]$.

\begin{thm} [Aistleitner--Borda--Hauke, 2022+]
\label{ABH}
Let $C>0$. 
For $\alpha \in \bR$, let $S(\alpha, N)$ denote 
the number of coprime pairs $(a,n) \in \bZ \times \bN$ 
such that
$$
n \le N, \qquad 
\left \vert \alpha - \: \frac{a}{n} \right \vert \leq \frac{\psi(n)}{n} .
$$
If 
\[
\Psi(N):= \sum_{n\leq N} 2\frac{\varphi (n)}{n}\psi(n) 
\]
is unbounded then, for almost all $\alp \in \bR$, we have
$$
S(\alpha, N) = \Psi(N) (1+O_C((\log \Psi(N))^{-C}))
$$
as $N\rightarrow \infty$.
\end{thm}

Our proof of Theorem \ref{FibreThm} also involves the theory of Bohr sets, as developed in our previous work \cite{Cho2018, CT2019}, which we use to verify the unboundedness condition in Theorem \ref{ABH}. In general $\Phi(n)$, as defined in \eqref{PhiDef}, will not lie in $[0,1/2]$, but the condition $\psi(n) < n^{-\kap}$ enables us to circumvent this and ultimately to apply Theorem \ref{ABH} to an allied approximating function. 

\begin{remark} We do not believe that the condition
\[
\psi(n) \in [0,1/2] \qquad (n \in \bN)
\]
is necessary in Theorem \ref{ABH}, though it is currently an assumption. It is necessary for many of the other theorems stated here, owing to the use of the distance to the nearest integer function $\| \cdot \|$. If one could relax this condition, then the condition that $\psi(n) < n^{-\kap}$ for all $n$ could be removed from Theorems \ref{MainThm} and \ref{FibreThm} but, instead of using $S_\bgam^\times(\balp,N,\psi)$, one would need to count pairs $(n,a_k) \in \bN \times \bZ$ such that
\[
n \le N, \qquad
\| n \alp_1 - \gam_1 \| \cdots \| n \alp_{k-1} - \gam_{k-1} \| \cdot |n \alp_k - a_k| < \psi(n).
\]
The latter counting function is greater than or equal to the former, so the reader should not be alarmed that our lower bound for it could far exceed $\Psi_k^\times(N)$ if $\psi$ were to be constant or decay very slowly.
\end{remark}

A natural `uniform' companion to
$S^\times_{\bgam}
(\balp,N,\psi)$ replaces $\psi(n)$ by $\psi(N)$ in the definition, giving rise to the counting function
\[
S^\times_{\bgam,\mathrm{unif}}
(\balp,N,\psi):=
\# \{
n \le N:
\| n \alp_1 - \gam_1 \| \cdots \| n \alp_k - \gam_k \| < \psi(N)
\}.
\]
When $\psi$ is not decaying too rapidly, 
lattice point counting can be successfully used
to obtain asymptotic formulas for 
$S^\times_{\mathbf{0},\mathrm{unif}}(\balp,N,\psi)$. 
We refer to the works of Widmer~\cite{Wid2017} and Fregoli~\cite{Fre2021}.

\subsection*{Notation.}
For complex-valued functions $f$ and $g$, we write $f \ll g$ or $f = O(g)$ if $|f| \le C|g|$ pointwise for some constant $C$, sometimes using a subscript to record dependence on parameters, and $f \asymp g$ if $f \ll g \ll f$. We write $f \sim g$ if $f/g \to 1$, and $f = o(g)$ if $f/g \to 0$.

\subsection*{Funding and acknowledgements}

NT was supported by a Schr\"odinger Fellowship of the Austrian Science Fund (FWF): project J 4464-N.
We thank Jakub Konieczny for raising the question, as well as for feedback on an earlier version of this manuscript, and we thank Christoph Aistleitner for a helpful conversation.

\section{Counting approximations on fibres}

In this section, we prove Theorem \ref{FibreThm}.
Fix $\eps > 0$ such that
\begin{equation} \label{epsrange}
10 k \sqrt{\eps} 
\le \min \left \{ \frac{1}{w} \: - \: \frac{k-2}{k-1}, \kap \right \} \in (0,1).
\end{equation}
We write
\[
\balp = (\alp_1,\ldots,\alp_{k-1}),
\qquad
\bgam = (\gam_1,\ldots,\gam_{k-1}).
\]
Define
\[
G = \{ n \in \bN: \| n \alp_i - \gam_i \| \ge n^{-\sqrt \eps} \quad (1 \le i \le k-1) \}
\]
and
\[
U_N(\balp,\bgam,\psi) = \sum_{\substack{n \le N \\ n \in G}} \frac{\varphi(n) \psi(n)}{n \| n \alp_1 - \gam_1 \| \cdots \| n \alp_{k-1} - \gam_{k-1} \|}.
\]
We showed in \cite[Equation (6.3)]{CT2019} that
\[
U_N(\balp,\bgam,\psi) \gg_\balp \: \tilde \Psi_k^\times(N),
\]
so the unboundedness assumption needed to apply Theorem \ref{ABH} to the approximating function
\[
n \mapsto \frac{\psi(n)}{ \| n \alp_1 - \gam_1 \| \cdots \| n \alp_{k-1} - \gam_{k-1} \| } 1_G(n) \in [0,1/2]
\]
is met. Thus, for almost all $\alp_k$, we have
\begin{equation} \label{ABHapp}
S^\times_{(\gam_1,\ldots,\gam_{k-1},0)}((\alp_1,\ldots,\alp_k),N,\psi) \gg U_N(\balp,\bgam,\psi).
\end{equation}

The implied constant in \cite[Equation (6.3)]{CT2019} was allowed to depend on $\balp$, however the following more uniform statement holds with essentially the same proof.

\begin{lemma} \label{uniform} Assume that $\psi$ is non-increasing.
Let $\balp = (\alp_1,\ldots,\alp_{k-1})$ be a real vector such that $ \ome^\times(\balp) = w < \frac{k-2}{k-1}$, and let $\bgam = (\gam_1,\ldots,\gam_{k-1}) \in \bR^{k-1}$. Then there exist $c = c(k,w,\kap) > 0$ and $N_0 = N_0(\balp)$ such that
\[
U_N(\balp,\bgam,\psi) \ge c \sum_{n \le N} \psi(n) (\log n)^{k-1} \qquad (N \ge N_0). 
\]
\end{lemma}

\begin{proof}
Recall that $\eps > 0$ satisfies \eqref{epsrange}.
First, we verify that the implicit constants in the
`inner structure' (\cite[Lemma 3.1]{CT2019})
and `outer structure'
(\cite[Lemma 3.2]{CT2019}) lemmas
depend only on $k$. That is, there exist positive constants $c_1 = c_1(k)$ and $c_2 = c_2(k)$ such that
if 
\begin{equation} \label{delrange}
N^{\sqrt \eps} \leq \delta_i \leq 1/2 \quad (1\leq i \leq k-1),
\qquad N\geq N_0
\end{equation}
then
\begin{equation}\label{eq: explicit order of magn}
c_1 \leq \frac{
 \# B_{\boldsymbol{\alpha}}^{\bzero} (N;\bdel)}
 {\del_1 \ldots \del_{k-1}N }
 \leq c_2,
\end{equation}
where
\[
B_\balp^{\bgam}(N;\bdel)
= \{ n \in \bZ: |n| \le N,
\| n \alp_i - \gam_i \| \le \del_i \: (1 \le i \le k-1) \}.
\]
The proofs of \cite[Lemma 3.1]{CT2019} and \cite[Lemma 3.2]{CT2019} are quite similar to one another, so we confine our discussion to the former. The only essential source of implied constants in its proof comes from the first finiteness theorem, and that implied constant only depends on $k$. The other implied constants that we introduced can easily be made absolute. For example, with $\lam,\lam_1$ as defined in \cite[\S 3.1]{CT2019}, the upper bound
\[
\| n \alp_1 \| \cdots \| n \alp_{k-1} \| 
\leq (\lambda_1/(10 \lambda))^{k-1} 
\del_1 \cdots \del_{k-1} \leq 
(\lambda_1/ \lambda)^{k-1} 
\]
and, for $n\geq N_{0}$, the lower bound
\[
\| n \alp_1 \| \cdots \| n \alp_{k-1} \| 
\geq n^{\eps-w} 
\geq (N \lam_1/(10\lam))^{\eps-w}.
\]
We thus have \eqref{eq: explicit order of magn}.

The construction of the base point $b_0$ in \cite[Section 3.2]{CT2019}
only requires $N_0$ to be large, and does not affect the implied constants as long as $N \ge N_0$. Thus, we have
\[
c_1 \leq \frac{
 \# B_{\boldsymbol{\alpha}}^{\bgam} (N;\bdel)}
 {\del_1 \cdots \del_{k-1}N }
 \leq c_2,
\]
subject to \eqref{delrange}.

With this at hand, the argument of \cite[Section 4]{CT2019} yields
$$
\sum_{n\in \hat{B}_{\boldsymbol{\alpha}}^{\bgam} (N;\bdel)}
\frac{\varphi(n)}{n}
\gg_{k,\eps} \: \del_1 \ldots \del_{k-1} N,
$$ 
where $\hat{B}_{\boldsymbol{\alpha}}^{\bgam} (N;\bdel)
= B_{\boldsymbol{\alpha}}^{\bgam} (N;\bdel)
\cap [N^{\sqrt{\varepsilon}},N]$.
The implied constant comes from Davenport's lattice point counting estimate \cite[Theorem 4.2]{CT2019} and the value of 
$\displaystyle \sum_{p \text{ prime}} p^{-1-\varepsilon}$,
and therefore only depends on $k,\eps$. 

Decomposing the range of summation into $(k-1)$-tuples of dyadic ranges for $(\del_1,\ldots,\del_{k-1})$, together with partial summation, as in \cite[Sections 5 and 6]{CT2019}, then gives
\[
U_N(\balp,\bgam,\psi) \gg_{k,\eps} \: \sum_{n \le N} \psi(n) (\log n)^{k-1} \qquad (N \ge N_0).
\]
Indeed, this process involves at least 
\[
c_3 (\log N)^{k-1}
\]
many $(k-1)$-tuples of dyadic ranges, where $c_3 = c_3(k,\eps) > 0$.

Finally, note that $\eps$ can be chosen to only depend on $k,w,\kap$.
\end{proof}

Combining Lemma \ref{uniform} with \eqref{ABHapp} completes the proof of Theorem \ref{FibreThm}.

\begin{appendix}

\section{Computing a volume}

Here we deduce Theorem \ref{WangYu} from \cite[Theorem 4.6]{Har1998} and the argument of Wang and Yu \cite[Section 1]{WY1981}. For $\lam > 0$, define
\[
\cB_k(\lam) = \{ \bx \in [0,1]^k: 0 \le x_1 \cdots x_k \le \lam \}
\]
and
\[
\cC_k(\lam) =  \{ \bx \in [0,1/2]^k: 0 \le x_1 \cdots x_k \le \lam \}
\]
By symmetry and \cite[Theorem 4.6]{Har1998}, for almost all $\balp \in \bR^k$, we have
\[
S_\bzero^\times(\balp, N, \psi) = T_k(N) + O(\sqrt{T_k(N)} (\log T_k(N))^{2+\eps}),
\]
where
\[
T_k(N) = 2^k \sum_{n \le N} \mu_k(\cC_k(\psi(n))).
\]
Thus, it remains to show that
\begin{equation} \label{vol}
T_k(N) \sim \Psi_k^\times(N) \qquad (N \to \infty).
\end{equation}

\begin{lemma} For $k \in \bN$ and $\lam > 0$, we have
\[
\mu_k(\cB_k(\lam)) =
\begin{cases}
1, &\text{if } \lam \ge 1 \\
\lam \sum_{s=0}^{k-1} \frac{(-\log \lam)^s}{s!},
&\text{if } 0 < \lam < 1.
\end{cases}
\]
\end{lemma}

\begin{proof} We induct on $k$. The base case is clear: $\mu_1(\cB_1(\lam)) = \min \{ \lam, 1\}$. Now let $k \ge 2$, and suppose the conclusion holds with $k-1$ in place of $k$. We may suppose that $0 < \lam < 1$. We compute that
\begin{align*}
\mu_k(\cB_k(\lam)) &= \int_0^1 \mu_{k-1} (\cB_{k-1}(\lam/x)) \d x = \lam + \int_\lam^1 \frac \lam x \sum_{s = 0}^{k-2} \frac{(\log(x/\lam))^s} {s!} \d x \\
&= \lam + \sum_{s=0}^{k-2} \frac{\lam}{s!} \int_1^{1/\lam} \frac{(\log y)^s}y \d y = \lam + \lam \sum_{s=0}^{k-2} \frac{ (- \log \lam)^{s+1}}{(s+1)!} \\
&= \lam \sum_{t=0}^{k-1} \frac{(-\log \lam)^t}{t!}.
\end{align*}
\end{proof}

In view of Schmidt's Theorem \ref{thm: Schmidt}, we may assume that $k \ge 2$. Now, as
\[
\mu_k(\cC_k(\lam)) 
= 2^{-k} \mu_k(\cB_k(2^k \lam)),
\]
and as $\psi(n) < 2^{-k}$ for large $n$, we have
\begin{align*}
T_k(N) &= O_{k,\psi}(1) + \sum_{n\le N} \psi(n) \sum_{s=0}^{k-1} \frac{(- \log (2^k \psi(n)))^s}{s!}  \\
&= \Psi_k^\times(N) + O_k(\Psi_{k-1}^\times(N)) + O_{k,\psi}(1).
\end{align*}
Since $\psi(n) \to 0$ as $n \to \infty$, we have $\Psi_{k-1}^\times(N) + 1 = o(\Psi_k^\times(N))$ and hence \eqref{vol}, completing the proof of Theorem \ref{WangYu}.
\end{appendix}

\providecommand{\bysame}{\leavevmode\hbox to3em{\hrulefill}\thinspace}


\begin{thebibliography}{50}

\bibitem{ABH}
C. Aistleitner, B. Borda, and M. Hauke, \emph{On the metric theory of approximations by reduced fractions: A quantitative Koukoulopoulos--Maynard theorem}, arXiv:2202.00936.

\bibitem{BHV2020} 
V. Beresnevich, A. Haynes, and S. Velani, \emph{Sums of reciprocals of fractional parts and multiplicative Diophantine approximation}, Mem. Amer. Math. Soc. \textbf{263} (2020).

\bibitem{Cho2018}
S. Chow, \emph{Bohr sets and multiplicative diophantine approximation}, Duke Math. J. \textbf{167} (2018), 1623--1642.

\bibitem{CT2019}
S. Chow and N. Technau,
\emph{Higher-rank Bohr sets and 
multiplicative diophantine approximation},
Compositio Math. \textbf{155} (2019), 2214--2233.

\bibitem{CT2021}
S. Chow and N. Technau, \emph{Littlewood and Duffin--Schaeffer-type problems in diophantine approximation}, Mem. Amer. Math. Soc., to appear.

\bibitem{DS1941}
R. J. Duffin and A. C. Schaeffer, 
\emph{Khintchine's problem in metric Diophantine approximation}, Duke Math. J. \textbf{8} (1941), 243--255.

\bibitem{Fre2021}
R. Fregoli, \emph{On a counting theorem for weakly admissible lattices},
Int. Math. Res. Not. \textbf{2021,} 7850--7884.

\bibitem{Gal1962}
P. X. Gallagher, \emph{Metric simultaneous diophantine approximation}, J. Lond. Math. Soc. \textbf{37} (1962), 387--390.

\bibitem{Har1998} 
G. Harman, \emph{Metric number theory}, London Math. Soc. Lecture Note Ser. (N.S.) \textbf{18,} Clarendon Press, Oxford, 1998.

\bibitem{Khi1924}
A. I. Khintchine, \emph{Einige S\"atze \"uber Kettenbr\"uche, mit Anwendungen auf die Theorie der Diophantischen Approximationen}, Math. Ann. \textbf{92} (1924), 115--125.

\bibitem{KM2020} D. Koukoulopoulos and J. Maynard,
\emph{On the Duffin-Schaeffer conjecture}, Ann. of Math. (2) \textbf{192} (2020), 251--307.

\bibitem{Sch1960} W. M. Schmidt, \emph{A metrical theorem in diophantine approximation}, Canad. J. Math. \textbf{12} (1960), 619--631.

\bibitem{WY1981}
Y. Wang and K. R. Yu, \emph{A note on some metrical theorems in Diophantine approximation},
Chinese Ann. Math. \textbf{2} (1981), 1--12.

\bibitem{Wid2017}
M. Widmer, \emph{Asymptotic diophantine approximation: the multiplicative case},
Ramanujan J. \textbf{43} (2017), 83--93.

\end{thebibliography}
\end{document}